\documentclass[11pt,british]{amsart}

\usepackage[T1]{fontenc}

\usepackage{babel,eucal,url,amssymb,stmaryrd,
enumerate,amscd,
}

\usepackage[pagebackref]{hyperref}

\theoremstyle{plain}
\newtheorem{thm}{Theorem}[section]
\newtheorem{lem}[thm]{Lemma}
\newtheorem{pro}[thm]{Proposition}

\theoremstyle{definition}
\newtheorem{defn}[thm]{Definition}

\theoremstyle{remark}

\newcommand{\Gtwo}{\ifmmode{{\rm G}_2}\else{${\rm G}_2$}\fi}

\date{\today}

\begin{document}

\title[Non-existence of flat paracontact metric structures...]{Non-existence of flat paracontact metric structures in dimension greater than or \\
equal to five}

\author{Simeon Zamkovoy}
\address[Zamkovoy]{University of Sofia "St. Kl. Ohridski"\\
Faculty of Mathematics and Informatics,\\
Blvd. James Bourchier 5,\\
1164 Sofia, Bulgaria} \email{zamkovoy@fmi.uni-sofia.bg}

\author{Vassil Tzanov}

\begin{abstract}
{An example of a three dimensional flat paracontact metric
manifold with respect to Levi-Civita connection is constructed. It
is shown that no such manifold exists for odd dimensions greater
than or equal to five.}

\end{abstract}

\maketitle \setcounter{tocdepth}{3} \tableofcontents

\section{Introduction}

The almost paracontact structure on pseudo-Riemannian manifold $M$
of dimension $(2n+1)$ is defined in \cite{K1} and the almost
paracomplex structure on $M^{(2n+1)}\times \mathbb{R}$ is
constructed. The properties of an almost paracontact metric
manifold and the gauge (conformal) transformations of a
paracontact metric manifold, i.e., transformations preserving the
paracontact structure, are studied in \cite{Z1}. Furthermore, in
this paper a canonical paracontact connection on a paracontact
metric manifold is defined. This connection is the paracontact
analogue of the (generalized) Tanaka-Webster connection. It is
shown that the torsion of the canonical paracontact connection
vanishes exactly when the structure is para-Sasakian and the gauge
transformation of its scalar curvature is computed.

An example of a paracontact structure with flat canonical
connection is the hyperbolic Heisenberg group \cite{I1}. The
paraconformal tensor gives a necessary and sufficient condition
for a $(2n+1)$-dimensional paracontact manifold to be locally
paracontact conformal to the hyperbolic Heisenberg group
\cite{I1}.

In this paper, we show that there is no flat, with respect to
Levi-Civita connection, paracontact metric structures in dimension
greater than or equal to five, whereas in dimension equal to three
there is.

\section{Preliminaries}
A (2n+1)-dimensional smooth manifold $M^{(2n+1)}$ has an
\emph{almost paracontact structure} $(\varphi,\xi,\eta)$ if it
admits a tensor field $\varphi$ of type $(1,1)$, a vector field
$\xi$ and a 1-form
$\eta$ satisfying the  following compatibility conditions 
\begin{eqnarray}
  \label{f82}
    & &
    \begin{array}{cl}
          (i)   & \varphi(\xi)=0,\quad \eta \circ \varphi=0,\quad
          \\[5pt]
          (ii)  & \eta (\xi)=1 \quad \varphi^2 = id - \eta \otimes \xi,
          \\[5pt]
          (iii) & \textrm{the tensor field $\varphi$ induces an almost paracomplex
                         structure (see \cite{K2})}
          \\[3pt]
                &  \textrm{on each fibre on the horizontal distribution $\mathbb D=Ker~\eta$.}
    \end{array}
\end{eqnarray}

Recall that an almost paracomplex structure on an 2n-dimensional
manifold is a (1,1)-tensor $J$ such that $J^2=1$ and the
eigensubbundles $T^+,T^-$ corresponding to the eigenvalues $1,-1$
of $J$ respectively, have dimensions equal to $n$. The Nijenhuis
tensor $N$ of $J$, given by
$N_{J}(X,Y)=[JX,JY]-J[JX,Y]-J[X,JY]+[X,Y],$ is the obstruction for
the integrability of the eigensubbundles $T^+,T^-$. If $N=0$ then
the almost paracomplex structure is called paracomplex or
integrable.

An immediate consequence of the definition of the almost
paracontact structure is that the endomorphism $\varphi$ has rank
$2n$, $\varphi \xi=0$ and $\eta \circ \varphi=0$, (see
\cite{B1,B2} for the almost contact case).

If a manifold $M^{(2n+1)}$ with $(\varphi,\xi,\eta)$-structure
admits a pseudo-Riemannian metric $g$ such that
\begin{equation}\label{con}
g(\varphi X,\varphi Y)=-g(X,Y)+\eta (X)\eta (Y),
\end{equation}
then we say that $M^{(2n+1)}$ has an almost paracontact metric structure and
$g$ is called \emph{compatible} metric. Any compatible metric $g$ with a given almost paracontact
structure is necessarily of signature $(n+1,n)$.

Setting $Y=\xi$, we have $\eta(X)=g(X,\xi).$

\smallskip

The fundamental 2-form
\begin{equation}\label{fund}
F(X,Y)=g(X,\varphi Y)
\end{equation}
is non-degenerate on the horizontal distribution $\mathbb D$ and
$\eta\wedge F^n\not=0$.
\begin{defn}
If $g(X,\varphi Y)=d\eta(X,Y)$ (where
$d\eta(X,Y)=\frac12(X\eta(Y)-Y\eta(X)-\eta([X,Y])$ then $\eta$ is
a paracontact form and the almost paracontact metric manifold
$(M,\varphi,\eta,g)$ is said to be $\emph{paracontact metric
manifold}$.
\end{defn}

\begin{defn}
An $r$-dimensional submanifold $N$ of $M^{(2n+1)}$ is said to be
an $\emph{integral}$ $\emph{submanifold}$ (of the horizontal
distribution $\mathbb D$) if and only if every tangent vector of
$N$ at every point $p$ of $N$ belongs to $\mathbb D$.
\end{defn}

\begin{defn}
An integral submanifold of dimension $r$ in $M^{(2n+1)}$ is said
to be a $\emph{maximal integral}$ $\emph{submanifold}$ if it is
not a pure subset of any other integral submanifold of dimension
$r$.
\end{defn}

Similarly to the contact metric case \cite{S1}, we may obtain the
following

\begin{pro}
Let $(M^{2n+1},\varphi,\eta,g)$ be a paracontact metric manifold.
Then the highest dimension of integral submanifold of the
horizontal distribution $\mathbb D$ is equal to $n$.
\end{pro}

The tensors $N^{(1)}$, $N^{(2)}$, $N^{(3)}$ and $N^{(4)}$ are
defined \cite{Z1} by
$$N^{(1)}(X,Y)=N_{\varphi}(X,Y) -2d\eta (X,Y)\xi,$$
$$N^{(2)}(X,Y)=(\pounds_{\varphi X}\eta)Y-(\pounds_{\varphi Y}\eta)X,$$
$$N^{(3)}(X)=(\pounds_{\xi}\varphi)X,$$
$$N^{(4)}(X)=(\pounds_{\xi}\eta)X.$$
Clearly the almost paracontact structure $(\varphi,\xi,\eta)$ is
normal if and only if these four tensors vanish.

In \cite{Z1} the following propositions are proven.
\begin{pro}\label{t2}
For an almost paracontact structure $(\varphi,\xi,\eta)$ the
vanishing of $N^{(1)}$ implies the vanishing $N^{(2)}$, $N^{(3)}$
and $N^{(4)}$;

For a paracontact structure $(\varphi,\xi,\eta,g)$, $N^{(2)}$ and
$N^{(4)}$ vanish. Moreover $N^{(3)}$ vanishes if and only if $\xi$
is a Killing vector field.
\end{pro}

\begin{pro}\label{l1}
For an almost paracontact metric structure $(\varphi,\xi,\eta,g)$,
the covariant derivative $\nabla\varphi$ of $\varphi$ with respect
to the Levi-Civita connection $\nabla$ is given by
\begin{gather}\label{f1}
2g((\nabla_X \varphi)Y,Z)= -dF(X,Y,Z)-dF(X,\varphi Y,\varphi
Z)-N^{(1)}(Y,Z,\varphi X)\\\nonumber
+N^{(2)}(Y,Z)\eta (X)-2d\eta (\varphi
Z,X)\eta (Y)+2d\eta (\varphi Y,X)\eta (Z).
\end{gather}
For a paracontact metric structure $(\varphi,\xi,\eta,g)$, the
formula \eqref{f1} simplifies to
\begin{equation}\label{f2}
2g((\nabla_X \varphi)Y,Z)= -N^{(1)}(Y,Z,\varphi X)-2d\eta (\varphi
Z,X)\eta (Y)+2d\eta (\varphi Y,X)\eta (Z)
\end{equation}
\end{pro}

\begin{lem}\label{l2}
On a paracontact metric manifold, h is a symmetric operator,
\begin{equation}\label{f3}
\nabla_X\xi=- \varphi X+\varphi hX,
\end{equation}
h anti-commutes with $\varphi$ and $trh=h\xi =0$.
\end{lem}

\section{Non-existence of flat paracontact metric structures in dimension greater than or
equal to five}

In this section we shall show that every paracontact metric
manifold of dimension greater than or equal to five must have some
curvature though not necessarily in the plane sections containing
$\xi$.
\begin{thm}\label{p1}
Let $M^{2n+1}$ be a paracontact manifold of dimension greater than
or equal to five. Then $M^{2n+1}$ cannot admit a paracontact
structure of vanishing curvature.
\end{thm}
\begin{proof}
The proof will be by contradiction. We let $(\varphi,\xi,\eta,g)$
denote the structure tensors of a paracontact metric structure and
assume that $g$ is flat. From \cite{Z1} we have that for a
paracontact metric structure
$$\frac12(R( \xi , X) \xi + \varphi R( \xi ,\varphi X) \xi)=
\varphi^2 X - h^2X,$$
where $h=\frac12\mathcal{L}_{\xi}\varphi$.
Thus if $g$ is flat $h^2=\varphi^2$ and hence $h\xi=0$ and
$rank(h)=2n$. The eigenvectors corresponding to the non-zero
eigenvalues of $h$ are orthogonal to $\xi$ and the non-zero
eigenvalues are $\pm 1$. Recall that
$d\eta(X,Y)=\frac12(g(\nabla_X\xi,Y)-g(\nabla_Y\xi,X))$ and that
for a paracontact metric structure
\begin{equation}\label{z1}
\nabla_X\xi=- \varphi X+\varphi hX.
\end{equation}
From $Lemma~\ref{l2}$ follows that whenever $X$ is an eigenvector
of eigenvalue $+1$, $\varphi X$ is an eigenvector of $-1$ and
vice-versa. Thus the paracontact distribution $\mathbb D$ is
decomposed into the orthogonal eigenspaces of $\pm 1$ which we
denote by $[+1]$ and $[-1]$.

We now show that the distribution $[+1]$ is integrable. If $X$ and
$Y$ are vector fields belonging to $[+1]$, equation $(\ref{z1})$
gives $\nabla_X\xi=0$ and $\nabla_Y\xi=0$. Thus since $M^{2n+1}$
is flat
$$0=R(X,Y)\xi=-\nabla_{[X,Y]}\xi=\varphi [X,Y]-\varphi h[X,Y];$$
but $\eta([X,Y])=-2d\eta(X,Y)=-2g(X,\varphi Y)=0$, so that
$h[X,Y]=[X,Y]$. Applying the same argument to $\xi$ and $X \in
[+1]$ we see that the distribution $[+1]\oplus[\xi]$ spanned by
$[+1]$ and $\xi$ is also integrable.

Since  $[+1]\oplus[\xi]$ is integrable, we can choose local
coordinates $(u^0,u^1,...,u^{2n})$ such that
$\frac{\partial}{\partial u^0},\frac{\partial}{\partial
u^1},...,\frac{\partial}{\partial u^n} \in [+1]\oplus[\xi]$. We
define local vector fields $X_i$, $i=1,...,n$ by
$X_i=\frac{\partial}{\partial
u^{n+i}}+\sum_{j=0}^{n}f_i^j\frac{\partial}{\partial u^j},$ where
the $f_i^j$'s are functions chosen so that $X_i \in [-1]$. Note
$X_1,...,X_n$ are $n$ linearly independent vector fields spanning
$[-1]$. Clearly $[\frac{\partial}{\partial u^k},X_i]\in
[+1]\oplus[\xi]$ for $k=0,...,n$ and hence $\xi$ is parallel along
$[\frac{\partial}{\partial u^k},X_i]$. Therefore using
$(\ref{z1})$ and the vanishing curvature
$$0=\nabla_{[\frac{\partial}{\partial u^k},X_i]}\xi=\nabla_{\frac{\partial}{\partial u^k}}\nabla_{X_i}\xi-\nabla_{X_i}\nabla_{\frac{\partial}{\partial u^k}}\xi=-2\nabla_{\frac{\partial}{\partial u^k}}\varphi X_i$$
from which we have
\begin{equation}\label{z2}
\nabla_{\varphi X_j}\varphi X_i=0.
\end{equation}
Similarly, noting that $[X_i,X_j] \in [+1],$
$$0=R(X_i,X_j)\xi=\nabla_{X_i}\nabla_{X_j}\xi-\nabla_{X_j}\nabla_{X_i}\xi-\nabla_{[X_i,X_j]}\xi=-2\nabla_{X_i}\varphi X_j+2\nabla_{X_j}\varphi X_i$$
giving
\begin{equation}\label{z3}
\nabla_{X_i}\varphi X_j=\nabla_{X_j}\varphi X_i
\end{equation}
or equivalently
\begin{equation}\label{z4}
\varphi
[X_i,X_j]=-(\nabla_{X_i}\varphi)X_j+(\nabla_{X_j}\varphi)X_i.
\end{equation}
Using equations $(\ref{z1})$ and $(\ref{z2})$ we obtain
$$0=R(X_i,\varphi X_j)\xi=-\nabla_{[X_i,\varphi X_j]}\xi=\varphi [X_i,\varphi X_j]-\varphi h[X_i,\varphi X_j]$$
from which
$$g([X_i,\varphi X_j],X_k)=g(h[X_i,\varphi X_j],X_k)=g([X_i,\varphi X_j],hX_k)=-g([X_i,\varphi X_j],X_k)$$
and hence
\begin{equation}\label{z5}
g([X_i,\varphi X_j],X_k)=0.
\end{equation}
Using the formula $(\ref{f2})$ and equations $(\ref{z2})$,
$(\ref{z4})$ and $(\ref{z5})$ we have
$$2g((\nabla_{X_i}\varphi)X_j,X_k)=-N^{(1)}(X_j,X_k,\varphi X_i)=-g([X_j,X_k],\varphi X_i)=$$
$$=-g((\nabla_{X_j}\varphi)X_k,X_i)+g((\nabla_{X_k}\varphi)X_j,X_i).$$
From $F=d\eta,$ we obtain
$\sigma_{i,j,k}g((\nabla_{X_i}\varphi)X_j,X_k)=0.$ Thus our
computation yields $g((\nabla_{X_i}\varphi)X_j,X_k)=0$. Similarly
$$2g((\nabla_{X_i}\varphi)X_j,\varphi X_k)=-N^{(1)}(X_j,\varphi X_k,\varphi
X_i)=-g([X_j,\varphi X_k],\varphi X_i)-g([\varphi X_j,X_k],\varphi
X_i)=$$
$$=-g(\nabla_{X_j}\varphi X_k-\nabla_{\varphi X_k}X_j-\nabla_{X_k}\varphi X_j+\nabla_{\varphi X_j}X_k,\varphi X_i)$$
which vanishes by equations $(\ref{z2})$ and $(\ref{z3})$. Finally
$$2g((\nabla_{X_i}\varphi)X_j,\xi)=-N^{(1)}(X_j,\xi,\varphi X_i)=-g(\varphi^2[X_j,\xi],\varphi X_i)+2d\eta(\varphi X_j,X_i)=$$
$$=-4g(X_i,X_j).$$
Thus for any vector fields $X$ and $Y$ in $[-1]$ on a paracontact
metric manifold such that $\xi$ is annihilated by the curvature
transformation,
\begin{equation}\label{z6}
(\nabla_X\varphi)Y=-2g(X,Y)\xi.
\end{equation}
Note that equation $(\ref{z4})$ now gives $[X_i,X_j]=0$.

Analogously, we obtain
\begin{equation}\label{z7}
2g(\nabla_{\varphi X_i}X_j,X_k)=2g((\nabla_{\varphi
X_i}\varphi)X_j,\varphi X_k)=0.
\end{equation}
Therefore by equation $(\ref{z5})$, we get
$$g(\nabla_{X_i}X_j,\varphi X_k)=-g(X_j,\nabla_{X_i}\varphi X_k)=-g(X_j,[X_i,\varphi X_k])=0.$$
It is trivial that $g(\nabla_{X_i}X_j,\xi)=0$ and hence we obtain
$\nabla_{X_i}X_j \in [-1]$.

Differentiating equation $(\ref{z6})$ we have
$$\nabla_{X_k}\nabla_{X_i}\varphi X_j-(\nabla_{X_k}\varphi)\nabla_{X_i}X_j-\varphi \nabla_{X_k}\nabla_{X_i}X_j=$$
$$=-2X_k(g(X_i,X_j))\xi+4g(X_j,X_i)\varphi X_k.$$
Taking the inner product with $\varphi X_l,$ having in mind
equation $(\ref{z6})$ and that $\nabla_{X_i}X_j \in [-1]$, we
obtain
\begin{equation}\label{z8}
g(\nabla_{X_k}\nabla_{X_i}\varphi X_j,\varphi
X_l)+g(\nabla_{X_k}\nabla_{X_i}X_j,X_l)=-4g(X_j,X_i)g(X_k,X_l)
\end{equation}
Interchanging $i$ and $k$, $i\not=k$ and subtracting we have
$$g(X_i,X_j)g(X_k,X_l)-g(X_k,X_j)g(X_i,X_l)=0$$
by virtue of the flatness and $[X_i,X_j]=0$.

Setting $i=j$ and $k=l$ we have
$$g(X_i,X_i)g(X_k,X_k)-g(X_i,X_k)g(X_i,X_k)=0$$
contradicting the linear independence of $X_i$ and $X_k$.
\end{proof}

Note that in the proof of our theorem, the vanishing of
$R(X,Y)\xi$ is enough to obtain the decomposition of the
paracontact distribution into $\pm 1$ eigenspaces of the operator
$h=\frac12\mathcal{L}_{\xi}\varphi$. Moreover $R(X,Y)\xi=0$ for
$X$ and $Y$ in $[+1]$ is sufficient for the integrability of
$[+1]$. Thus we have the following
\begin{thm}\label{p2}
Let $M^{2n+1}$ be a paracontact manifold with paracontact metric
structure $(\varphi,\xi,\eta,g)$. If the sectional curvatures of
all plane sections containing $\xi$ vanish, then the operator $h$
has rank $2n$ and the paracontact distribution is decomposed into
the $\pm 1$ eigenspaces of $h$. Moreover if $R(X,Y)\xi=0$ for $X,
Y \in [+1]$, $M$ admits a foliation by $n-$dimensional integral
submanifolds of the paracontact distribution along which $\xi$ is
parallel.
\end{thm}
From $Theorem~\ref{p1}$ and $Theorem~\ref{p2}$ we obtain following
\begin{thm}\label{p3}
Let $M^{2n+1}$ be a paracontact metric manifold and suppose that\\
$R(X,Y)\xi=0$ for all vector fields $X$ and $Y$. Then locally
$M^{2n+1}$ is the product of a flat $(n+1)$-dimensional manifold
and $n$-dimensional manifold of negative constant curvature equal
to $-4$.
\end{thm}
\begin{proof}
We noted in $Theorem~\ref{p1}$ proof that $[X_i,X_j]=0$ so that
the distribution $[-1]$ is also integrable and hence we may take
$X_i=\frac{\partial}{\partial u^{n+i}}$. Moreover locally
$M^{2n+1}$ is the product of an integral submanifold $M^{n+1}$ of
$[+1]\oplus[\xi]$ and an integral submanifold $M^{n}$ of $[-1]$.
Since $\{\varphi X_i,\xi \}$ is a local basis of tangent vector
fields on $M^{n+1}$, equation $(\ref{z2})$ and $R(X,Y)\xi=0$ show
that $M^{n+1}$ is flat.

Now $\nabla_{\varphi X_i}X_j=0$, by equation $(\ref{z7})$
$g(\nabla_{\varphi X_i}X_j,X_k)=0$, by equation $(\ref{z2})$
$g(\nabla_{\varphi X_i}X_j,\varphi X_k)=0$, and $g(\nabla_{\varphi
X_i}X_j,\xi)=0$ which is trivial. Interchanging $i$ and $k$ in
equation $(\ref{z8})$ and subtracting we have
$$R(X_k,X_i,\varphi X_j,\varphi X_l)+R(X_k,X_i,X_j,X_l)=$$
$$=-4(g(X_i,X_j)g(X_k,X_l)-g(X_k,X_j)g(X_i,X_l)).$$
Using $\nabla_{\varphi X_i}X_j=0$ and $[\varphi X_i,\varphi
X_j]=0$ we see that $R(X_k,X_i,\varphi X_j,\varphi
X_l)=$\\$=R(\varphi X_j,\varphi X_l,X_k,X_i)=0$ and hence
$$R(X_k,X_i,X_j,X_l)=-4(g(X_i,X_j)g(X_k,X_l)-g(X_k,X_j)g(X_i,X_l))$$
completing proof.
\end{proof}

\section{Flat associated metrics on $\mathbb{R}^3_1$}\label{flat}

In dimension $3$ it is easy to construct flat paracontact
structures. For example, consider $\mathbb{R}^3_1$ with
coordinates $(x^1,x^2,x^3)$. The $1$-form
$\eta=\frac12(ch(x^3)dx^1+sh(x^3)dx^2)$ is a paracontact form. In
this case $\xi=2(ch(x^3)\frac{\partial}{\partial
x^1}-sh(x^3)\frac{\partial}{\partial x^2})$ and the metric $g$
whose components are $g_{11}=-g_{22}=g_{33}=\frac14$ gives flat
paracontact metric structure. Following the proof of the
$Theorem~\ref{p1}$, we see that $\frac{\partial}{\partial x^3}$
spans the distribution $[-1]$ and $sh(x^3)\frac{\partial}{\partial
x^1}+ch(x^3)\frac{\partial}{\partial x^2}$ spans the distribution
$[+1]$. Geometrically $\xi$ is parallel along $[+1]$ and rotates
(and hence the paracontact distribution $\mathbb{D}$ rotates) as
we move parallel to the $x^3$-axis.

We can now find a flat associated metric on $\mathbb{R}^3_1$ for
the standard paracontact form $\eta_0=\frac12(dz-ydx)$. Consider
the diffeomorphism $f: \mathbb{R}^3_1 \rightarrow \mathbb{R}^3_1$
given by
$$x^1=zch(x)-ysh(x)$$
$$x^2=zsh(x)-ych(x)$$
$$x^3=-x$$
Then $\eta_0=f^*\eta$ and the pseudo-Riemannian metric $g_0=f^*g$
is a flat associated metric for the paracontact form $\eta_0$.

\textbf{Acknowledgement} Simeon Zamkovoy acknowledges support from
the European Operational programm HRD through contract
BGO051PO001/07/3.3-02/53 with the Bulgarian Ministry of Education.
He also was partially supported by Contract 082/2009 with the
University of Sofia "St. Kl. Ohridski".

\bibliographystyle{hamsplain}






\end{document}